\documentclass[a4paper,10pt]{amsart}
\usepackage{amsmath,amsfonts,amssymb,amsthm, amsfonts, amscd, url}
\usepackage[utf8]{inputenc}
\usepackage[english]{babel}
\usepackage{url}

\usepackage[cmtip,arrow]{xy}
\usepackage{pb-diagram,pb-xy}

\swapnumbers 
\theoremstyle{plain}
\newtheorem{thm}{Theorem}[section]
\newtheorem{prop}[thm]{Proposition}
\newtheorem{lemma}[thm]{Lemma}

\theoremstyle{remark}
\theoremstyle{definition}
\newtheorem{rem}[thm]{Remark}
\newtheorem{rems}[thm]{Remarks}
\newtheorem{remdef}[thm]{Remark-Definition}
\newtheorem{remsdefs}[thm]{Remarks-Definitions}

\newtheorem{defi}[thm]{Definition}

\newtheorem{notas}[thm]{Notations}

\usepackage[english]{babel}

\title[Isometries of some compact Lie groups]{The full group of isometries of some compact Lie groups endowed with a bi-invariant metric}

\author{Alberto Dolcetti \and Donato Pertici}

\begin{document}

\parindent 0pt
\selectlanguage{english}

\maketitle

\vspace*{-0.2in}

\begin{center}
{\scriptsize Dipartimento di Matematica  e Informatica, Viale Morgagni 67/a, 50134 Firenze, ITALIA

\vspace*{0.07in}

alberto.dolcetti@unifi.it, \  http://orcid.org/0000-0001-9791-8122

\vspace*{-0.03in}

donato.pertici@unifi.it,  \   http://orcid.org/0000-0003-4667-9568}

\end{center}


\vspace*{-0.1in}

\begin{abstract}
We describe the full group of isometries of absolutely simple, compact, connected real Lie groups,  of $S\mathcal{O}(4)$ and of $U(n)$, endowed with suitable bi-invariant Riemannian metrics.
\end{abstract}


{\small \tableofcontents}

\renewcommand{\thefootnote}{\fnsymbol{footnote}}

\renewcommand{\thefootnote}{\arabic{footnote}}
\setcounter{footnote}{0}

\vspace*{-0.3in}

{\small {\scshape{Keywords.}} (absolutely simple, compact, connected) real Lie group, Lie algebra, Killing metric, Frobenius metric, (special) orthogonal group, compact symplectic group, (special) unitary group. 

\smallskip

{\small {\scshape{Mathematics~Subject~Classification~(2020):}} 53C35, 22E15.

{\small {\scshape{Grants:}}
This research has been partially supported by GNSAGA-INdAM (Italy).

\bigskip

\section*{Introduction}\label{intro}

In this paper we describe the full group of isometries of some classes of \emph{real Lie groups}, endowed with suitable bi-invariant Riemannian metrics: the \emph{Killing metric} both on any \emph{ absolutely simple,}\footnote{In the present paper, a real Lie group is said to be \emph{absolutely simple} if the complexification of its Lie algebra is a simple Lie algebra.} \emph{compact, connected} Lie group and on the \emph{special orthogonal group} $S\mathcal{O}(4)$, and also the metric induced, on the \emph{unitary group} $U(n)$, by the flat Frobenius metric of $M_n(\mathbb{C})$.

In \cite{DoPe2015} and in \cite{DoPe2018} we have already studied  another relevant example of (semi-Riemannian) metric: the so-called trace metric, which is bi-invariant on $GL_n(\mathbb{R})$ and on its Lie subgroups. Some techniques, used in the present paper, have been developed in those papers and in \cite{DoPe2019}, \cite{DoPe2020}, \cite{DoPe2021}.

\smallskip

Given any Lie group $G$, the Killing form of its Lie algebra extends, on the whole $G$, to a bi-invariant symmetric $(0, 2)$-tensor, denoted by $\mathcal{K}$, and called \emph{Killing tensor} of $G$. 

Some further properties of $G$ give some relevant consequences. For instance, as well-known, $G$ is semi-simple if and only if $\mathcal{K}$ (and also $-\mathcal{K}$) is a semi-Riemannian metric on $G$ (Cartan's criterion) and, if $G$ is semi-simple and compact, then the tensor $-\mathcal{K}$ is a Riemannian metric on $G$, we call \emph{Killing metric} of $G$.  Furthermore, if $G$ is connected, compact and simple, then $(G, - \mathcal{K})$ is a globally symmetric Riemannian manifold with non-negative sectional curvature and, moreover, if $G$ is also \emph{absolutely simple}, then $(G, - \mathcal{K})$ is an Einstein manifold.
The Killing tensor of $G$ is more than an example of bi-invariant tensor on $G$, indeed if $G$ is connected and absolutely simple, then every bi-invariant real $(0,2)$-tensor on $G$ is a constant multiple of $\mathcal{K}$. 
These results are essentially matter of Section \ref{preliminari}.

\smallskip

Section \ref{ris-generali} is devoted to the general result of this paper:

\smallskip

{\bf Theorem \ref{teor-gen}} Let $G$ be a absolutely simple, compact, connected real Lie group and
let $- \mathcal{K}$ be its Killing metric.
Then $F: (G, -\mathcal{K}) \to (G, -\mathcal{K})$ is an isometry if and only if 
there exist an element $a \in G$ and an automorphism $\Phi$ of the Lie group $G$ such that
either $F = L_a \circ \Phi$ or $F = L_a \circ \Phi \circ j$, where $L_a$ is the left translation associated to $a$ and $j$ is the inversion map.

\smallskip

Many classical groups satisfy all conditions of the above Theorem, precisely:
the \emph{special orthogonal groups} $S \mathcal{O}(n)$, with $n \ge 3$ and $n \ne 4$,
the \emph{special unitary groups} $SU(n)$, with $n \ge 2$,
the \emph{compact symplectic groups} $Sp(n)$, with $n \ge 1$.

A careful analysis of the automorphisms of each group allows us to deduce the complete lists of the isometries of $(G, -\mathcal{K})$, where $G$ is one of the previous classical groups, (Theorem \ref{teor-gen-n}).

\smallskip

The manifold $(S\mathcal{O}(4), -\mathcal{K})$ is not included in the previous result:  indeed $S\mathcal{O}(4)$ is semi-simple but not simple, however $-\mathcal{K}$ is still a Riemannian metric on it. Section \ref{caso-SO4} is devoted to this particular case. The key points are the following: $(S\mathcal{O}(4), -\mathcal{K})$ is isometric to the Lie group $\dfrac{SU(2) \times SU(2)}{\{ \pm (I_2, I_2)\}}$ (endowed with its Killing metric) and the natural covering projection of  $SU(2) \times SU(2)$ (endowed with the product of the Killing metrics) onto the previous quotient, is clearly a local isometry.  All isometries of $SU(2) \times SU(2)$ are obtained by means of the analysis of Section \ref{ris-generali}, via a classical result of De Rham. Since these ones project as isometries of the quotient, we can obtain the main result of the section:

\smallskip

{\bf Theorem. \ref{isom-SO-4}} 
The isometries of $(S \mathcal{O}(4), - \mathcal{K})$ are precisely the following maps: 

$X \to AXB$, \ \ 
$X \to AX^{T}B$, \ \ 
$X \to A \, \tau(X)B$,  \ \ 
$X \to A \, \tau(X)^{T}B$,  \ \ 
where $A,B$ are matrices both in $S\mathcal{O}(4)$ or both in $\mathcal{O}(4) \setminus S\mathcal{O}(4)$ (and $\tau$ is a suitable map constructed by means of the \emph{Cayley's factorization} of $S \mathcal{O}(4)$).

\smallskip

Finally, Section \ref{caso-U-n} is devoted to $U(n)$, endowed with the bi-invariant Riemannian metric $\phi$, restriction to $U(n)$ of the flat Frobenius metric of $M_n(\mathbb{C})$. This metric is not multiple of the Killing tensor, because $U(n)$ is not semi-simple (and so its Killing tensor is degenerate). Analogously to Section \ref{caso-SO4}, we get a covering map (which is also a local isometry) from $SU(n) \times \mathbb{R}$ (endowed with a suitable product metric) onto $(U(n), \phi)$. This allows to get the following main results of this Section, with a difference between the cases $n=2$ and $n \ge 3$, due to the fact that all isometries of $SU(2) \times \mathbb{R}$ project as isometries of $(U(2), \phi)$, whereas, for $n \ge 3$, $SU(n) \times \mathbb{R}$ has more types of isometries, not all projecting as maps of $U(n)$.

\smallskip

{\bf Theorem \ref{isom-U-2}.} The isometries of $(U(2), \phi)$ are precisely the following maps:

$X \to AXB$, \ \ 
$X \to AX^*B$, \ \ 
$X \mapsto \dfrac{AXB}{\det(X)}$, \ \ 
$X \to\det(X) AX^*B$, \ \ 
with $A,B \in U(2)$.

\smallskip

{\bf Theorem \ref{isom-U-n}.} The isometries of $(U (n), \phi)$, with $n \ge 3$,  are precisely the following maps: 

$X \to AXB$, \ \ 
$X \to AX^{*}B$, \ \ 
$X \to A\overline{X}B$,\ \ 
$X \to AX^{T}B$, \ \ 
with $A, B \in U(n)$.

\bigskip

{\bf Acknowledgement.} We wish to express our gratitude to Fabio Podest\`{a} for his help and for many discussions about the matter of this paper.

\section{Notations and preliminary facts.}\label{preliminari}

\begin{notas}\label{notazioni} \ \\
In this paper we will use many standard notations from the matrix theory, which should be clear from the context, among these: $M_n(\mathbb{R})$ for the vector space of real square matrices, $\mathcal{O}(n)$ for the group of real orthogonal matrices, $S\mathcal{O}(n)$ for the the group of real special orthogonal matrices, $Sp(n)$ for the compact symplectic group, $M_n{(\mathbb{C})}$ for the vector space of complex square matrices, $U(n)$ for the group of  unitary matrices, $SU(n)$ for the group of special unitary matrices (all matrices of order $n$).
If $A$ is a matrix, then $A^T$, $A^{-1}$, $\overline{A}$, $A^{*} := \overline{A}^T$ denote its transpose, its inverse (when it exists), its conjugate and its transpose conjugate. $I_n$ is the identity matrix of order $n$ and ${\bf i} \in \mathbb{C}$ is the unit imaginary number.

\smallskip

The basic notations and notions on real Lie groups and algebras are the following:
\begin{itemize}
\item  $G$ is a real Lie group with identity $e$ and inversion map $j: x \mapsto x^{-1}$, $\mathfrak{g}$ is its Lie algebra (identified with its tangent space at $e$), $\exp: \mathfrak{g} \to G$ is the \emph{exponential map} and $Aut(G)$ denotes the Lie group of all (smooth) automorphisms of $G$;

\item if $\mathfrak{g}$ is a real Lie algebra, $\mathfrak{g}^{\mathbb{C}}:= \mathfrak{g} \oplus \mathbf{i} \mathfrak{g} = \mathfrak{g} \otimes_{\mathbb{R}} \mathbb{C}$ will denote its \emph{complexification}, which turns to be a complex Lie algebra, having $\mathfrak{g}$ as real subalgebra;

\item if $\mathfrak{h}$ is a complex Lie algebra, $\mathfrak{h}^{\mathbb{R}}$ will denote its \emph{realification}, i. e. $\mathfrak{h}^{\mathbb{R}}$ is simply $\mathfrak{h}$, regarded as real Lie algebra;

\item for every $a \in G$, $L_a$ and $R_a$ are, respectively, the \emph{left and right translations} in $G$ associated to $a$ and $C_a := L_a \circ R_{a^{-1}}$ is the \emph{inner automorphism} of $G$ associated to $a$;

\item for every $a \in G$, $Ad_a$ is the automorphism of $\mathfrak{g}$, defined as the differential at $e$ of $C_a$. It is well-known that $\exp \circ Ad_a = C_a \circ \exp$;

\item $\mathcal{K}$ is the left-invariant symmetric $(0, 2)$-tensor on the whole $G$, extending the \emph{Killing form} of $\mathfrak{g}$, so that the Killing form of $\mathfrak{g}$ agrees with $\mathcal{K}_e$. We call $\mathcal{K}$ the \emph{Killing tensor} of the Lie group $G$.
\end{itemize}
\end{notas}

\begin{lemma}\label{prime-isom}
The Killing tensor $\mathcal{K}$ of the Lie group $G$ is bi-invariant on $G$ and it is preserved by every $\phi \in Aut(G)$ and by the inversion map $j$ (i. e. $\phi^* (\mathcal{K}) = \mathcal{K}$ and $j^* (\mathcal{K}) = \mathcal{K}$).
\end{lemma}

\begin{proof}
$\mathcal{K}_e$ is invariant with respect to all automorphisms of $\mathfrak{g}$, hence the left-invariant tensor $\mathcal{K}$ is preserved by all smooth automorphisms of $G$ (in particular by all inner automorphisms) and so $\mathcal{K}$ is right-invariant too.  For the assertion on $j$ see, for instance,  \cite[pp.\,147-148]{Helg1978}.
\end{proof}

\begin{remsdefs}
We say that a (finite dimensional) Lie algebra $\mathfrak{g}$ is \emph{simple}, if it is non-abelian and  has no ideals except $0$ and $\mathfrak{g}$; while we say that $\mathfrak{g}$ is \emph{semi-simple}, if it splits into the direct sum of simple Lie algebras; by the well-known Cartan's criterion, $\mathfrak{g}$ is semi-simple if and only if its Killing form is non-degenerate (see for instance \cite{SSLie}).

A Lie group is said to be \emph{simple} (respectively \emph{semi-simple}), if its Lie algebra is simple (respectively semi-simple). Hence a simple Lie group is semi-simple too.

Note that if $G$ is a semi-simple Lie group, then $(G, \mathcal{K})$ and $(G, -\mathcal{K})$ are \emph{semi-Riemannian} manifolds. We refer to $-\mathcal{K}$ (the opposite of the Killing tensor $\mathcal{K}$) as the \emph{Killing metric} of the (semi-simple) Lie group.
\end{remsdefs}

\begin{prop}\label{propr-gen}
Let $G$ be a semi-simple connected Lie group. Then

a) the geodesics of the semi-Riemannian manifold $(G, -\mathcal{K})$ are precisely the curves of the form $t \mapsto x \exp(tv)$, for every  $t \in \mathbb{R}$, with $x$ generic in $G$ and $v$ generic in the Lie algebra $\mathfrak{g}$ of $G$ (so $(G, -\mathcal{K})$ is \emph{geodesically complete}); 

b) the Levi-Civita connection $\nabla$ of $(G, -\mathcal{K})$ is the $0$-connection of Cartan-Schouten, defined by $\nabla_X (Y) :=\dfrac{1}{2}[X, Y]$, where $X$ and $Y$ are left-invariant vector fields on $G$;

c) the curvature tensor of type $(1, 3)$ of $(G, -\mathcal{K})$ is  

$R_{XY}Z:= \nabla_{[X, Y]}Z - [\nabla_X, \nabla_Y] Z = \dfrac{1}{4}[[X,Y], Z]$, 

where $X$, $Y$, $Z$ are left-invariant vector fields on $G$; 

d) the curvature tensor of type $(0, 4)$ of $(G, -\mathcal{K})$ is the bi-invariant tensor, defined by 

$R_{XYZW}:= -\mathcal{K}(R_{XY}Z, W) =-\dfrac{1}{4}\mathcal{K}([X,Y], [Z, W])$, 

where $X$, $Y$, $Z$, $W$ are left-invariant vector fields on $G$.
\end{prop}

\begin{proof}
Part (a), part (b) and part (c) follow directly from the results contained in 

\cite[p.\,148 and p.\,548-550]{Helg1978} (our tensor $R$ is the opposite of the corresponding tensor of \cite{Helg1978}).

Part (c) implies that $R_{XYZW} = -\dfrac{1}{4} \mathcal{K}([[X,Y], Z], W)$. By the skew-symmetry, with respect to the Killing form, of every operator $ad_v: x \to [v,x]$ (see for instance \cite{Killing}), we have $\mathcal{K}([[X,Y], Z], W) = \mathcal{K}([X,Y], [Z, W])$ and this concludes (d).
\end{proof}

\begin{remdef}\label{def-compl}
We say that a real Lie group $G$ is a \emph{complex Lie group}, if it possesses a complex analytic structure, compatible with the real one, such that multiplication and inversion are holomorphic. It is known that a real Lie group $G$ with Lie algebra $\mathfrak{g}$ is complex if and only if there exists a complex Lie algebra $\mathfrak{h}$ such that $\mathfrak{h}^{\mathbb{R}} = \mathfrak{g}$ (see  \cite[Prop.\,1.110 p.\,95]{Knapp2002}).
\end{remdef}

\begin{lemma}\label{lemma-compl}
Let $G$ be a real Lie group and let $\mathfrak{g}$ be its Lie algebra with $\mathfrak{g}^{\mathbb{C}}$ as complexification.
Then 
the complex Lie algebra $\mathfrak{g}^{\mathbb{C}}$ is simple if and only if
$G$ is simple and not complex.
\end{lemma}

\begin{proof}
It follows from \cite[Thm.\,6.94 p.\,407]{Knapp2002}, remembering that if $\mathfrak{g}^{\mathbb{C}}$ is a simple complex Lie algebra, then $\mathfrak{g}$ is a simple real Lie algebra.
\end{proof}

\begin{defi}
We say that a real Lie group is \emph{absolutely simple} if it is simple and not complex or, equivalently by Lemma \ref{lemma-compl}, if the complexification of its Lie algebra is a simple, complex Lie algebra.
\end{defi}

A standard consequence of Schur's Lemma is the following

\begin{prop}\label{multKilling}
Let $G$ be a real Lie group and assume that $G$ is connected and absolutely simple. Then every bi-invariant real $(0,2)$-tensor on $G$ is a constant multiple of the Killing metric $- \mathcal{K}$ of $G$.
\end{prop}

\begin{lemma}\label{semisimple compact}
Let $G$ be a real Lie group and assume that $G$ is semi-simple and compact. Then the Killing tensor $\mathcal{K}$ of $G$ is negative definite at every point (i.e. the Killing metric $- \mathcal{K}$ is a Riemannian metric on $G$).
\end{lemma}

\begin{proof}
It follows from \cite[Prop.\,6.6(i) p.\,132, Cor.\,6.7 p.\,133]{Helg1978}.
\end{proof}

\begin{rem}\label{defU}
Let $G$ be a simple, compact, connected, real Lie group and let $\mathfrak{g}$ be its Lie algebra; denote by $\Delta$ the \emph{diagonal} of $G \times G$ and by $Z$ the \emph{center} of $G$.

$Z$ is a closed subgroup of $G$ and it is finite. Indeed the center of $\mathfrak{g}$ is zero (since $G$ is simple, see \cite[Cor.\,6.2 p.\,132]{Helg1978}). Since the Lie algebra of $Z$ agrees with the center of $\mathfrak{g}$, then $Z$ is a discrete subgroup of the compact group $G$ and therefore $Z$ is finite.

Now we denote by $\mathcal{U}$ the semisimple compact connected Lie group defined by
$\mathcal{U}:= \dfrac{G \times G}{(Z \times Z) \cap \Delta}$ 
and consider the map 
$$T: \mathcal{U} \times G \to G, \ \ \ T (\{(g, h)\}, x) =  gxh^{-1}$$ 
where $\{(g, h)\}$ is the class of $(g, h)$ in $\dfrac{G \times G}{(Z \times Z) \cap \Delta}$.

\smallskip

$T$ is an effective and transitive left action of $\mathcal{U}$ on $G$ and its isotropy subgroup at the identity is 
$\widehat{\Delta} :=\dfrac{\Delta}{(Z \times Z) \cap \Delta}$. 
Therefore $G$ is diffeomorphic to the homogeneous space $\dfrac{\mathcal{U}}{\widehat{\Delta}}$.
Moreover, for every $\{(g, h)\} \in \mathcal{U}$, the map $x \mapsto T (\{(g, h)\}, x)$ is an isometry with respect to $-\mathcal{K}$ (and to $\mathcal{K}$).
Finally the pair $(\mathcal{U}, \widehat{\Delta})$ is a \emph{Riemannian symmetric pair} (in the sense of \cite[p.\,209]{Helg1978}) with \emph{involutive automorphism} given by $\sigma(\{(g, h)\})=\{(h, g)\}$.
\end{rem}

\begin{prop}\label{G-glob-simm}
Let $G$ be a simple, compact, connected, real Lie group and let $-\mathcal{K}$ be its Killing metric.

Then $(G, - \mathcal{K})$ is a globally symmetric Riemannian manifold with non-negative sectional curvature; furthermore every connected component of the Lie group of its isometries is diffeomorphic to $\dfrac{G \times G}{(Z\times Z) \cap \Delta}$, where $Z$ is the center of $G$ and $\Delta$ is the diagonal of $G \times G$.

Moreover, if $G$ is absolutely simple too, then $(G, - \mathcal{K})$ is an Einstein manifold.
\end{prop}

\begin{proof}
By \cite[Prop.\,3.4 p.\,209]{Helg1978},\  $(G, - \mathcal{K})$ is a globally symmetric Riemannian manifold, via Remark \ref{defU}. By Proposition \ref{propr-gen} (d), the sectional curvature of the space generated by two left-invariant and $\mathbb{R}$-independent vector fields $X, Y$ of $G$, agrees with

$- \dfrac{1}{4} \mathcal{K}([X, Y],[X, Y])$, which is non-negative and equal to $0$ if and only if $[X, Y] =0$ . The assertion about the connected components of the Lie group of the isometries, follows from \cite[Thm.\,4.1 (i) p.\,243]{Helg1978} and from the fact that in a Lie group all connected components are diffeomorphic to the component containing the identity.

The last statement is a consequence of Proposition \ref{multKilling}, taking into account that the Ricci tensor of $(G, - \mathcal{K})$ is bi-invariant.
\end{proof}

\begin{rem}
For further details and information on Lie groups with bi-invariant metrics, we refer, for instance, to \cite[Ch.\,2]{AlBet2015}.
\end{rem}

\section{Isometries of a compact Lie group}\label{ris-generali}

\begin{lemma}\label{lemma-F} Let $\mathfrak{g}$ be a real Lie algebra, whose complexification  $\mathfrak{g}^{\mathbb{C}}$ is a simple, complex Lie algebra and let $L$ be an isometry with respect to the Killing form $\mathcal{B}$ of $\mathfrak{g}$, such that
$L([v, w]) = [v, L(w)]$ for every $v, w \in \mathfrak{g}$.

Then  $L = \pm Id_{\mathfrak{g}}$. 
\end{lemma}

\begin{proof}
 The killing form $\mathcal{B}^{\mathbb{C}}$ of $\mathfrak{g}^{\mathbb{C}}$ is the extension by $\mathbb{C}$-linearity of the Killing form $\mathcal{B}$ of $\mathfrak{g}$; by $\mathbb{C}$-linearity too, $L$ can be extended to a map $L^{\mathbb{C}}: \mathfrak{g}^{\mathbb{C}} \to \mathfrak{g}^{\mathbb{C}}$ which is an isometry, with respect to the Killing form $\mathcal{B}^{\mathbb{C}}$ of $\mathfrak{g}^{\mathbb{C}}$, satisfying again the analogous condition $L^{\mathbb{C}}([v, w]) = [v, L^{\mathbb{C}}(w)]$ for every $v, w \in \mathfrak{g}^{\mathbb{C}}$. Let $\lambda \in \mathbb{C}$ be an eigenvalue of $L^{\mathbb{C}}$ and let $V_{\lambda} \ne \{0\}$ be the corresponding eigenspace. If $v \in \mathfrak{g}^{\mathbb{C}}$ and $w \in V_{\lambda}$, then $L^{\mathbb{C}}([v, w]) =[v, \lambda w] = \lambda [v, w]$, so $[v, w] \in V_{\lambda}$, which turns out to be a non-zero ideal of $\mathfrak{g}^{\mathbb{C}}$ and therefore $V_{\lambda}= \mathfrak{g}^{\mathbb{C}}$, i. e. $L^{\mathbb{C}}= \lambda \, Id_{_{\mathfrak{g}^{\mathbb{C}}}}$. Since $L^{\mathbb{C}}$ is an isometry with respect to the Killing form $\mathcal{B}^{\mathbb{C}}$, which is non-degenerate by Cartan's criterion, the map $L^{\mathbb{C}}$ agrees with $\pm Id_{_{\mathfrak{g}^{\mathbb{C}}}}$, so that $L = \pm Id_{\mathfrak{g}}$.
\end{proof}

\begin{prop}\label{ident-fissa}
Let $G$ be an absolutely simple, compact, connected, real Lie group and 
Let $- \mathcal{K}$ be its Killing metric.

Then $F: (G, - \mathcal{K})  \to (G, - \mathcal{K})$ is an isometry, fixing the identity  $e \in G$, if and only if there exists an automorphism $\Phi$ of the Lie group $G$ such that
either $F = \Phi$ or $F = \Phi \circ j$, where $j$ is the inversion map.
\end{prop}

\begin{proof}

Lemma \ref{prime-isom} implies that the automorphisms and the inversion map of the Lie group $G$ are isometries with respect to $- \mathcal{K}$, fixing $e$.

For the converse, let $\mathcal{J}$ be the group of isometries of $(G, - \mathcal{K})$, let $\mathcal{J}_e$ be the corresponding  subgroup of isotropy at $e$ and let $\mathcal{J}^0$, $\mathcal{J}_e^0$ be their connected components containing the identity.
In Remark \ref{defU}, we observed that $(\mathcal{U}, \widehat{\Delta})$ is a Riemannian symmetric pair and so, by \cite[Thm.\,4.1(i) p.\,243]{Helg1978}, we have $\mathcal{J}^0 \simeq \mathcal{U}$ (as Lie groups). 
From this we get that $dim(\mathcal{J})=dim(\mathcal{J}^0) = dim(\mathcal{U}) = 2 \, dim(G)$ and therefore 

$dim(\mathcal{J}_{e}^0) =  dim(\mathcal{J}_{e}) =dim(\mathcal{J}) - dim(G) = dim(G)$.

Let us consider the adjoint representations of $G$ and of its Lie algebra $\mathfrak{g}$, denoted by $Ad: G \to GL(\mathfrak{g})$ and by $ad: \mathfrak{g} \to \mathfrak{gl}(\mathfrak{g})$, respectively; we indicate with $Ad(G)$ and with $ad(\mathfrak{g})$ their images. 
Note that $Ad(G)$ is a closed Lie subgroup of $GL(\mathfrak{g})$ and $ad(\mathfrak{g})$ is its Lie algebra; moreover, since the kernel of the map $ad$ agrees with the center of $\mathfrak{g}$ and this last is zero, we get that $ad: \mathfrak{g} \to ad(\mathfrak{g})$ is an isomorphism of Lie algebras; this implies that $Ad(G)$ and $G$ have the same dimension.

Let us consider also the representation
$d: \mathcal{J}_{e} \to GL(\mathfrak{g})$ defined as the differential at $e$ of every element of $\mathcal{J}_{e}$. 
By \cite[Prop.\,62 p.\,91]{O'N1983}, $d$ is a faithful representation and so: $d(\mathcal{J}_e^0) = (d(\mathcal{J}_e))^0$ (the component of the image $d(\mathcal{J}_e)$ containing the identity). Hence $dim(d(\mathcal{J}_e))^0 = dim(\mathcal{J}_{e}^0) = dim (G)$.
Since $G$ is connected, we have the inclusion $Ad(G)  \subseteq (d(\mathcal{J}_e))^0$. Now these manifolds have the same dimension, hence, by the \emph{Theorem of invariance of domain}, $Ad(G)$ is open in $(d(\mathcal{J}_e))^0$; moreover, $Ad(G)$ is compact and $(d(\mathcal{J}_e))^0$ is connected and this allows to get that $Ad(G) = (d(\mathcal{J}_e))^0$.

For any fixed $F \in \mathcal{J}_e$, the previous equality gives: $d F \, Ad(G) \, d F^{-1} = Ad(G)$. Hence there exists a unique automorphism $\alpha$ of $Ad(G)$ such that: 

(*) \ \ \ \ \ \ \ \ \ \ \ \ \ \ \ \ \ \ \ $d F \circ Ad_X \circ d F^{-1} = \alpha(Ad_X)$, \  for every $X \in G$.

\smallskip

We denote by $\exp: \mathfrak{g} \to G$ and by $\widehat{\exp}: ad(\mathfrak{g}) \to Ad(G)$ the two usual exponential maps. It is well-known that $Ad \circ \exp = \widehat{\exp} \circ ad$ (see for instance \cite[Thm.\,3.28 p.\,60]{Ha2015}).

For every $t \in\mathbb{R}$ and every $v \in \mathfrak{g}$, the (*) above becomes:

(**) \ \ \ \ \ \ \ \ \ \ \ \ \ \ \ \ \ \ \ \ \ $d F \circ Ad_{\exp(tv)} \circ d F^{-1} = \alpha(Ad_{\exp(tv)})$.

\smallskip

Now let $\widetilde{\alpha}$ be the unique automorphism of $ad(\mathfrak{g})$ such that $\alpha \circ \widehat{\exp} = \widehat{\exp} \circ \widetilde{\alpha}$.
The map $\overline{\alpha} := ad^{-1} \circ \widetilde{\alpha} \circ ad$ is an automorphism of the Lie algebra $\mathfrak{g}$, satisfying $Ad  \circ \exp \circ  \, \overline{\alpha} = \alpha  \circ  Ad \circ \exp$. 
Hence, for every $t \in\mathbb{R}$ and every $v \in \mathfrak{g}$, the (**) becomes:

(***)  \ \ \ \ \ \ \ \ \ \ \ \ \ \ \ \ \ \ \ \ $d F \circ Ad_{\exp(tv)} \circ d F^{-1} = Ad_{\exp(t \overline{\alpha}(v) )}$.

\smallskip

Now, if we derive the identity (***), with respect to $t$, for $t=0$, we get:
$$
d F \circ ad_v \circ d F^{-1} = ad_{\overline{\alpha}(v)}
.$$
Since $ad_v(w)= [v, w]$ for every $v, w \in \mathfrak{g}$ and remembering that $\overline{\alpha}$ is an automorphism of the Lie algebra $\mathfrak{g}$, we get 
$
dF([v, w]) = [\overline{\alpha}(v), dF(w)]= \overline{\alpha}([v, \overline{\alpha}^{-1}(dF(w))])
$
and so
$
(\overline{\alpha}^{-1} \circ dF)([v, w]) = [v, (\overline{\alpha}^{-1} \circ dF)(w))],$ for every $v, w \in \mathfrak{g}$.
Note that $dF$ and $\overline{\alpha}$ are both isometries of $\mathfrak{g}$, with respect to its Killing form; moreover, since $G$ is absolutely simple, its Lie algebra $\mathfrak{g}$ satisfies the hypotheses of Lemma \ref{lemma-F}, thus we obtain $dF = \pm \overline{\alpha}$.

Let $\pi: \widetilde{G} \to G$ be the universal covering group of $G$ and let $\widetilde{F}: \widetilde{G} \to \widetilde{G}$ be such that $F \circ \pi = \pi \circ \widetilde{F}$, with $\widetilde{F}(\widetilde{e}) = \widetilde{e}$, where $\widetilde{e}$ is the identity of $\widetilde{G}$; from this we get $ \widetilde{F}_* =  \pi_*^{-1} \circ dF \circ  \pi_* =  \pi_*^{-1} \circ (\pm \overline{\alpha}) \circ  \pi_*$ , where $ \widetilde{F}_*$ ,\  $\pi_*$ denote the differentials at the identity $\widetilde{e}$ of $\widetilde{F}$ and $\pi$, respectively. 
If we denote by $\beta$ the automorphism of the Lie algebra $\widetilde{\mathfrak{g}}$ of $\widetilde{G}$, given by $\beta=  \pi_*^{-1} \circ \overline{\alpha} \circ \pi_*$, we can write $ \widetilde{F}_* = \pm \beta$.

By \cite[Thm.\,3.27 p.\,101]{Warner1983}, there exists a unique automorphism $\Psi$ of the simply connected Lie group $\widetilde{G}$, whose
differential at the identity $\widetilde{e}$, $\Psi_*$, agrees with $\beta$. Hence $\widetilde{F}_* = \pm \Psi_*$.

Since $\Psi$ is an automorphism of $\widetilde{G}$, it is an isometry of $(\widetilde{G}, - \widetilde{\mathcal{K}})$, where $-\widetilde{\mathcal{K}}$ is the Killing metric of $\widetilde{G}$ (remember Lemma \ref{prime-isom}).
It is easy to check that $\pi: (\widetilde{G}, - \widetilde{\mathcal{K}}) \to (G, - {\mathcal{K}})$ is a local isometry and this implies that $\widetilde{F}: (\widetilde{G}, - \widetilde{\mathcal{K}}) \to (\widetilde{G}, - \widetilde{\mathcal{K}})$ is an isometry too.

If $\widetilde{F}_* = \Psi_*$, then $\widetilde{F} =\Psi$ (see, for instance, \cite[Prop.\,62 p.\,91]{O'N1983}) and hence $F \circ \pi = \pi \circ \Psi$. The surjectivity of $\pi$, together with the fact that $\pi$ and $\Psi$ are Lie group homomorphisms, implies that $F$ is a (bijective) endomorphism of $G$. This allows to conclude that $F \in Aut(G)$.

Suppose now that $\widetilde{F}_* = - \Psi_*$. We denote by $\widetilde{j}$ the inversion map of $\widetilde{G}$ and by $ \widetilde{j}_*$ its differential at the identity $\widetilde{e}$ . By Lemma \ref{prime-isom}, \  $\widetilde{j}$ is an isometry of $(\widetilde{G}, -\widetilde{\mathcal{K}})$, furthermore $\widetilde{j}_*$ agrees with the opposite of the identity map (see, for instance,  \cite[p.\,147]{Helg1978}).

Now $\widetilde{F}_* =  \widetilde{j}_* \circ \Psi_* = (\widetilde{j} \circ \Psi)_*$ and, arguing as in the previous case, we get that $\widetilde{F} = \widetilde{j} \circ  \Psi$ and so,  $F \circ \pi = \pi \circ \widetilde{j} \circ  \Psi = j \circ  \pi  \circ \Psi$. Hence $j \circ F \circ \pi = \pi  \circ \Psi$ and, as above, we obtain that $\Phi:= j \circ F \in Aut(G)$; therefore we conclude that $F = j \circ \Phi = \Phi \circ j$, with $\Phi \in Aut(G)$.
\end{proof}

\begin{thm}\label{teor-gen}
Let $G$ be an absolutely simple, compact, connected real Lie group and
let $- \mathcal{K}$ be its Killing metric.
Then $F: (G, -\mathcal{K}) \to (G, -\mathcal{K})$ is an isometry if and only if 
there exist an element $a \in G$ and an automorphism $\Phi$ of the Lie group $G$ such that
either $F = L_a \circ \Phi$ or $F = L_a \circ \Phi \circ j$, where $L_a$ is the left translation associated to $a$ and $j$ is the inversion map.
\end{thm}

\begin{proof}
Note that $L_a \circ \Phi$ and $L_a \circ \Phi \circ j$ are both isometries, because they are compositions of isometries (remember again Lemma \ref{prime-isom}).

The converse follows from Proposition \ref{ident-fissa}, because, for $a= F(e)$, \ $L_{a^{-1}} \circ F$ is an isometry fixing the identity $e \in G$.
\end{proof}

\begin{rem}
As known, relevant examples of absolutely simple, compact, connected, real Lie groups are:

\begin{itemize}
\item \emph{the special orthogonal group} $S \mathcal{O}(n)$, $n \ge 3$, $n \ne 4$;

\item \emph{the special unitary group} $SU(n)$, $n \ge 2$;

\item \emph{the compact symplectic group} $Sp(n)$, $n \ge 1$.
\end{itemize}

The automorphisms of $S \mathcal{O}(n)$, with $n \ge 3$ odd, of $SU(2)$ and of $Sp(n)$, with $n \ge 1$, are precisely the \emph{inner automorphisms} of the corresponding group.

Furthermore the automorphisms of $S \mathcal{O}(n)$, with $n \ge 6$ even, are precisely the maps $X \mapsto AXA^T$, with $A \in \mathcal{O}(n)$.

Finally the automorphisms of $SU(n)$, with $n \ge 3$, are the inner automorphisms and all the maps $X \mapsto C \overline{X} C^*$, where $C \in SU(n)$.

From these facts and  from Theorem \ref{teor-gen} we can easily get the following

\end{rem}

\begin{thm}\label{teor-gen-n}\ \\
a) The isometries of $(S \mathcal{O}(n), - \mathcal{K})$, with $n \ge 3$ odd, are precisely the following maps:

$X \to AXB$ and  $X \to AX^{T}B$, with $A, B \in S \mathcal{O}(n)$.

\smallskip

b) The isometries of $(S \mathcal{O}(n), - \mathcal{K})$, with $n \ge 6$ even, are precisely the following maps:

$X \to AXB$ and $X \to AX^{T}B$, with $A,B$ both in $S\mathcal{O}(n)$ or both in $\mathcal{O}(n) \setminus S\mathcal{O}(n)$.

\smallskip

c) The isometries of $(S U (2), - \mathcal{K})$  are precisely the following maps:

$X \to AXB$ and $X \to AX^{*}B$, with $A, B \in SU(2)$.

\smallskip

d) The isometries of $(S U (n), - \mathcal{K})$, with $n \ge 3$,  are precisely the following maps:

 $X \to AXB$, \  $X \to AX^{*}B$, \  $X \to A\overline{X}B$ and  $X \to AX^{T}B$, with $A, B \in SU(n)$.

\smallskip

e) The isometries of $(Sp(n), - \mathcal{K})$,  with $n \ge 1$, are precisely the following maps:

$X \to AXB$ and  $X \to AX^{-1}B$, with $A, B \in Sp(n)$.
\end{thm}

\begin{rem}
The Lie groups of isometries of $(S \mathcal{O}(n), - \mathcal{K})$, with $n \ge 3$ odd, of isometries of $(S U (2), - \mathcal{K})$, and of isometries of $(Sp(n), - \mathcal{K})$, with $n \ge 1$, have two connected components, while the Lie groups of isometries of $(S \mathcal{O}(n), - \mathcal{K})$, with $n \ge 6$ even, and of isometries of $(SU(n), - \mathcal{K})$, with $n \ge 3$, have four connected components.
\end{rem}

\begin{rem}\label{anche-Frob}
If $G$ is one of the groups  $S \mathcal{O}(n)$, $n \ge 3$ and $n \ne 4$, $SU(n)$, $n \ge 2$, or $Sp(n)$, $n \ge 1$, then 
$\mathcal{K}_A(X, Y) = c \cdot tr(A^{-1}XA^{-1}Y)$, for some strictly positive constant $c$, for every $A \in G$ and for every  $X, Y \in T_A(G)$ (as we can deduce, for instance, from \cite[Ex.\,6.19 p.\,129]{Sepa2007}). 

We denote by ${\phi}$ the (flat) \emph{Frobenius hermitian metric} of $M_m(\mathbb{C})$ ($m \ge 2$), defined by ${\phi}(A, B) = Re(tr(A B ^*))$, for every $A, B \in M_m(\mathbb{C})$. To simplify the notations, we denote also by ${\phi}$ its restriction to each submanifold $N$ of $M_m(\mathbb{C})$ and we call it \emph{Frobenius metric} of $N$. It is just a computation that, if $A \in U(m)$, then the maps $L_A$ and $R_A$, are isometries of $(M_m(\mathbb{C}), \phi)$ and, therefore, the Frobenius metric of $U(m)$ is bi-invariant.
Moreover,  arguing as in \cite[Recall\,4.1]{DoPe2018}, it is simple to verify that the expression of the Frobenius metric $\phi $ of $U(m)$ is the following: $\phi_A(X, Y) = -tr(A^* X A^* Y)$, for every $A \in U(m)$ and for every $X, Y \in T_A(U(m))$.

In each of the above cases, $G$ is a (closed) Lie subgroup of $U(n)$ or of $U(2n)$, i.e. $G$ is a submanifold of some $U(m)$ ($m \ge 2$); hence, on $G$, the metric $\phi$ is bi-invariant and $\phi = - \dfrac{1}{c} \mathcal{K}$ (with $c >0$).
Therefore, if $G$ is one of the above groups, then Proposition \ref{propr-gen}, Proposition \ref{G-glob-simm} and Theorem \ref{teor-gen-n} hold also with $\phi$ instead of $- \mathcal{K}$. 
\end{rem}

\begin{rem}
Parts (a) and (b) of Theorem \ref{teor-gen-n} \ can be compared with an analogous result, obtained in \cite[Thm.\,1]{AAH2013}.
\end{rem}

\section{Isometries of $S\mathcal{O}(4)$}\label{caso-SO4}

\begin{rem}
By Lemma \ref{semisimple compact}, the Killing metric of the semi-simple compact Lie group $S\mathcal{O}(4)$ is a Riemannian metric on $S\mathcal{O}(4)$. It is easy to check that the Killing form of the \emph{special orthogonal Lie algebra} $\mathfrak{so}(4)$, evaluated at $U, V$, agrees with $2 \, tr( U , V )$ (this extend to the case $n=2$ the formula ($3$) of \cite[Ex.\,6.19  p.\,129]{Sepa2007}). Hence the  Killing metric $-\mathcal{K}$ of $S\mathcal{O}(4)$ agrees with the double of the Frobenius metric $\phi$ of $S\mathcal{O}(4)$. 
Therefore, for the Lie group $S\mathcal{O}(4)$, Proposition \ref{propr-gen} holds for $\phi$ as well as for $-\mathcal{K}$. 
However, in \cite[Prop.\,4.3]{DoPe2018}, we have already proved that $(S\mathcal{O}(4), \phi)$ (and so, also $(S\mathcal{O}(4),- \mathcal{K})$) is an Einstein globally symmetric Riemannian manifold with non-negative sectional curvature.
\end{rem}

\begin{remsdefs}\label{Cayley} \ \\
a) The map
$\rho: \mathbb{C} \to M_2(\mathbb{R})$, given by $\rho(z) :=  
\begin{pmatrix}
Re(z) & -Im(z) \\ 
Im(z) & Re(z)
\end{pmatrix}
$
is a monomorphism of $\mathbb{R}$-algebras between $\mathbb{C}$ and $M_2(\mathbb{R})$. 

More generally, for any $h \ge 1$, we still denote  by $\rho$ the monomorphism of $\mathbb{R}$-algebras: $M_h(\mathbb{C}) \to M_{2h}(\mathbb{R})$, which maps the $h \times h$ complex matrix $Z=(z_{ij})$ to the $(2h) \times (2h)$ block real matrix  $(\rho(z_{ij}))$, having $h^2$ blocks of order $2 \times 2$. We refer to $\rho$ as the \emph{decomplexification map} of $M_h(\mathbb{C})$ into $M_{2h}(\mathbb{R})$.

It is known that, for every $Z \in M_h(\mathbb{C})$, the map $\rho$ satisfies:

$tr(\rho(Z)) = 2 Re(tr(Z)))$, $det(\rho(Z)) = |det(Z)|^2$ and $\rho(Z^*) = \rho(Z)^T$.

For simplicity, we still denote by $\rho$ all its restrictions to any subset of $M_h(\mathbb{C})$. Hence, for instance, $\rho(U(h)) = \rho(M_h(\mathbb{C})) \cap S\mathcal{O}({2h})$ is a Lie subgroup of $S\mathcal{O}({2h})$ (isomorphic to $U(h)$) and, in particular, $\rho(SU(2))$ is a Lie subgroup of $S\mathcal{O}(4)$, isomorphic to $SU(2)$.

\smallskip

b) We consider the matrix $J= J^T=J^{-1} \in \mathcal{O}(4)$, defined by
$J:=
\begin{pmatrix}
-1 & 0 & 0 & 0 \\ 
0 & 1 & 0 & 0 \\
0 & 0 & 1 & 0 \\
0 & 0 & 0 & 1
\end{pmatrix}
$ and the Lie subgroup of $S\mathcal{O}(4)$, conjugate to $\rho(SU(2))$ in $\mathcal{O}(4)$, defined by $J \rho(SU(2)) J$. It is easy to check that $\rho(SU(2)) \cap \big(J \rho(SU(2)) J\big) = \{ \pm I_4\}$ and that $X$ commutes with $JYJ$, for every $X, Y \in \rho(SU(2))$. Moreover, it is known that every matrix of $S\mathcal{O}(4)$ has a \emph{Cayley's factorization} as commutative product of a matrix of $\rho(SU(2))$ and of a matrix of $J \rho(SU(2)) J$ and that such factorization is unique up to the sign of both matrices (see, for instance, \cite[Thm.\,3.2]{BMCC2016} and also \cite{Mebius2005}, \cite{Thomas2014}, \cite{PerThom2017})

.
\smallskip

c) Let us consider $X=\rho(X_1) \, [J\rho(X_2)J]$ (with $X_1, X_2 \in SU(2)$) a matrix in $S\mathcal{O}(4)$, together with its Cayley's factorization. 
The map 
$\tau: S\mathcal{O}(4) \to S\mathcal{O}(4)$, given by

$X=\rho(X_1) \, [J\rho(X_2)J] \mapsto \tau(X):=\rho(X_1) \, [J \rho(X_2) J]^T = \rho(X_1) \, [J \rho(X_2^*) J]$, 

is well-defined and bijective; moreover $\tau^2 = Id$ \ and \ $\tau \circ j = j \circ \tau$ (where $j$ is the inversion map (i.e. the transposition map) of $S\mathcal{O}(4)$).

\smallskip

d) The map $\widehat{\chi}: SU(2) \times SU(2) \to S\mathcal{O}(4)$, defined by $\widehat{\chi}(X, Y)= \rho(X) J \rho(Y)J$, is an epimorphism of Lie groups, whose kernel is $\lbrace\pm (I_2, I_2)\rbrace$. Then the map $\widehat{\chi}$ induces a Lie group isomorphism $\chi : \dfrac{SU(2) \times SU(2)}{\{ \pm (I_2, I_2)\}} \to S\mathcal{O}(4)$.

Then $(S\mathcal{O}(4), - \mathcal{K})$ is a Riemannian manifold isometric to $(\dfrac{SU(2) \times SU(2)}{\{ \pm (I_2, I_2)\}}, - \mathcal{K}')$, where $-\mathcal{K}'$ is the Killing metric of the Lie group $\dfrac{SU(2) \times SU(2)}{\{ \pm (I_2, I_2)\}}$.

\smallskip 

e) The Killing tensor of $SU(2) \times SU(2)$\  is \ $\mathcal{K}_2 \times \mathcal{K}_2$, where $\mathcal{K}_2$ denotes the Killing tensor of $SU(2)$. We denote by 
$\sigma: SU(2) \times SU(2) \to SU(2) \times SU(2)$ the map which interchanges the two factors of $SU(2) \times SU(2)$. By a classical result due to de Rham (see \cite[Thm.\,III p.\,341]{deRham1952}), the isometries of $(SU(2) \times SU(2), - (\mathcal{K}_2 \times \mathcal{K}_2))$ are precisely the maps $ \psi_1 \times \psi_2 :  (X, Y) \to (\psi_1(X), \psi_2(Y))$ and the maps 

$ (\psi_1 \times \psi_2) \circ \sigma : (X, Y) \to (\psi_1(Y), \psi_2(X))$, where $\psi_1, \psi_2$ are isometries of $(SU(2), - \mathcal{K}_2)$. 
In particular, the map $\sigma$ in an isometry of $(SU(2) \times SU(2), - (\mathcal{K}_2 \times \mathcal{K}_2))$.

\smallskip

From these facts and from Theorem 2.5 (c), if we denote by $j$ the inversion map of $SU(2)$, we get the following
\end{remsdefs}

\begin{prop}\label{SU2xSU2}
The isometries of $(SU(2) \times SU(2), - (\mathcal{K}_2 \times \mathcal{K}_2))$ are precisely the maps of the form:

$(L_{A_1} \circ R_{A_2}) \times  (L_{B_1} \circ R_{B_2})$;

$(L_{A_1} \circ R_{A_2} \circ j) \times (L_{B_1} \circ R_{B_2})$;

$(L_{A_1} \circ R_{A_2})  \times (L_{B_1} \circ R_{B_2}  \circ j)$;

$(L_{A_1} \circ R_{A_2} \circ j) \times (L_{B_1} \circ R_{B_2}  \circ j)$;

$\big( (L_{A_1} \circ R_{A_2}) \times  (L_{B_1} \circ R_{B_2}) \big) \circ \sigma$;

$\big( (L_{A_1} \circ R_{A_2} \circ j) \times (L_{B_1} \circ R_{B_2}) \big) \circ \sigma$;

$\big( (L_{A_1} \circ R_{A_2})  \times (L_{B_1} \circ R_{B_2}  \circ j) \big) \circ \sigma$;

$\big( (L_{A_1} \circ R_{A_2} \circ j) \times (L_{B_1} \circ R_{B_2}  \circ j) \big) \circ \sigma$

where $A_1, A_2, B_1, B_2$ are generic elements of $SU(2)$.

In particular the isometries of $(SU(2) \times SU(2), - (\mathcal{K}_2 \times \mathcal{K}_2))$, fixing the identity $(I_2, I_2)$, are the previous ones, with $A_1^*= A_2$ and $B_1^*= B_2$.
\end{prop}

\begin{prop}\label{covering}
Let $\pi: SU(2) \times SU(2) \to \dfrac{SU(2) \times SU(2)}{\{ \pm (I_2, I_2)\}}$ be the natural covering projection. 

If $\Psi$ is an isometry of $(\dfrac{SU(2) \times SU(2)}{\{ \pm (I_2, I_2)\}}, -\mathcal{K}')$, fixing the identity of the group, then there exists a unique isometry $\widetilde{\Psi}$ of $((SU(2) \times SU(2)), - (\mathcal{K}_2 \times \mathcal{K}_2))$, fixing the identity $(I_2, I_2)$, such that $\Psi \circ \pi = \pi \circ \widetilde{\Psi}$.

Conversely, if $\widetilde{\Psi}$ is an isometry of $((SU(2) \times SU(2)), - (\mathcal{K}_2 \times \mathcal{K}_2))$, fixing the identity $(I_2, I_2)$, then there exists a unique isometry $\Psi$ of $(\dfrac{SU(2) \times SU(2)}{\{ \pm (I_2, I_2)\}}, -\mathcal{K}')$, fixing the identity of the group, such that $\Psi \circ \pi = \pi \circ \widetilde{\Psi}$.

\end{prop}

\begin{proof}
Let $\Psi$ be an isometry of $(\dfrac{SU(2) \times SU(2)}{\{ \pm (I_2, I_2)\}}, -\mathcal{K}')$, fixing the identity of the group. Since $SU(2) \times SU(2)$ is simply connected, there exists a unique homeomorphism 

$\widetilde{\Psi}: SU(2) \times SU(2) \to SU(2) \times SU(2)$, fixing the identity$(I_2, I_2)$ and such that $\Psi \circ \pi = \pi \circ \widetilde{\Psi}$. Since $\pi$ is a local isometry from $(SU(2) \times SU(2), - (\mathcal{K}_2 \times \mathcal{K}_2))$ onto $(\dfrac{SU(2) \times SU(2)}{\{ \pm (I_2, I_2)\}}, -\mathcal{K}')$, the map $\widetilde{\Psi}$ is an isometry of $(SU(2) \times SU(2), - (\mathcal{K}_2 \times \mathcal{K}_2))$.

For the converse, we denote by $\mu$ the isometry of $(SU(2) \times SU(2), - (\mathcal{K}_2 \times \mathcal{K}_2))$ defined by $\mu (A, B) = (-A, -B)$. From Theorem \ref{teor-gen-n} (c) and from Remarks-Definitions \ref{Cayley} (e), the map $\mu$ commutes with all isometries $\widetilde{\Psi}$ of $(SU(2) \times SU(2), - (\mathcal{K}_2 \times \mathcal{K}_2))$, fixing the identity of the group, and so, all these last project as isometries of the quotient.
\end{proof}

\begin{thm}\label{isom-SO-4}
The isometries of $(S \mathcal{O}(4), - \mathcal{K})$ are precisely the following maps:

$X \to AXB$, 

$X \to AX^{T}B$,

$X \to A \, \tau(X)B$, 

$X \to A \, \tau(X)^{T}B$, 

where $A,B$ are matrices both in $S\mathcal{O}(4)$ or both in $\mathcal{O}(4) \setminus S\mathcal{O}(4)$.
\end{thm}

\begin{proof}
By Propositions \ref{SU2xSU2} and \ref{covering}, all isometries of $(\dfrac{SU(2) \times SU(2)}{\{ \pm (I_2, I_2)\}}, -\mathcal{K}')$, fixing the identity, are obtained by projecting onto the quotient the following isometries of 

$((SU(2) \times SU(2)), - (\mathcal{K}_2 \times \mathcal{K}_2))$: 

$C_A \times C_B  $, \ $(C_A \times C_B)\circ (j \times id)$, \ $(C_A \times C_B) \circ (id \times j)$, \ $(C_A \times C_B) \circ (j \times j)$, \ $(C_A \times C_B) \circ \sigma$, \ $(C_A \times C_B) \circ (j \times id) \circ \sigma$, \ $(C_A \times C_B) \circ (id \times j) \circ \sigma$, \ $(C_A \times C_B) \circ (j \times j) \circ \sigma$,

with $A,B \in SU(2)$. Here $id$ and $j$ denote, respectively, the identity and the inversion map of $SU(2)$, whereas $C_X$ denotes, as usual, the inner automorphism of $SU(2)$ associated to a generic element $X$ of $SU(2)$.

By Remarks-Definitions \ref{Cayley} (d), the isometries of $(S \mathcal{O}(4), - \mathcal{K})$, fixing the identity $I_4$, are of the form $\chi \circ \Phi \circ \chi^{-1}$, where $\Phi$ is one of the above isometries of $(\dfrac{SU(2) \times SU(2)}{\{ \pm (I_2, I_2)\}}, -\mathcal{K}')$.

Standard computations show that $\chi \circ (C_A \times C_B) \circ \chi^{-1}= C_{_{\widehat{\chi}(A,B)}}$, for every $A, B \in SU(2)$, $\chi \circ (id \times j)\circ  \chi^{-1}= \tau$ (and so, $\tau$ is an isometry of $(S \mathcal{O}(4), - \mathcal{K})$), $\chi \circ (j \times id) \circ \chi^{-1} = \tau \circ \widehat{j} = \widehat{j} \circ \tau$, $\chi \circ (j \times j) \circ \chi^{-1} = \widehat{j}$, where  $\widehat{j}$ denotes the inversion map of $S\mathcal{O}(4)$, and $\chi \circ \sigma \circ \chi^{-1} = C_J$, where $J$ is the matrix of $\mathcal{O}(4) \setminus S \mathcal{O}(4)$, defined in Remarks-Definitions \ref{Cayley} (b). From this, we get that the complete list of the isometries of $(S \mathcal{O}(4), - \mathcal{K})$, fixing the identity $I_4$, is the following:
$C_M$, \ $C_M  \circ \widehat{j} \circ \tau$, \ $C_M \circ \tau$, \ $C_M \circ \widehat{j}$, 
where $M$ is a generic matrix of $\mathcal{O}(4)$.

To get the full group of isometry of $(S \mathcal{O}(4), - \mathcal{K})$, it suffices to compose these isometries with a left translation $L_A$, with $A \in S\mathcal{O}(4)$. This allows to conclude the proof.
\end{proof}

\begin{rem}
The full group of isometries of $(S\mathcal{O}(4), -\mathcal{K})$ has $8$ connected components, all diffeomorphic to $\dfrac{S\mathcal{O}(4) \times S\mathcal{O}(4)}{\{\pm(I_4, I_4)\}}$.
\end{rem}

\begin{rem}
Also Theorem \ref{teor-gen-n} can be compared with the analogous result, obtained in \cite[Thm.\,1]{AAH2013}, for $n=4$..
\end{rem}

\section{Isometries of $U(n)$}\label{caso-U-n}

In this Section we describe the full group of isometries of the Riemannian manifold $(U(n), \phi)$ ($n \ge 2$), where $\phi$ is  the Frobenius metric of $U(n)$, defined by $\phi_A(X, Y) = -tr(A^* X A^* Y)$, for every $A \in U(n)$ and for every $X, Y \in T_A(U(n))$. By the way, note that $\phi$ can be also obtained, by the Frobenius metric ${\phi}_0$ of  $S\mathcal{O}({2n})$, as $\phi = \dfrac{1}{2} \, \rho^*({\phi}_0)$, where $\rho$ is the decomplexification map of $U(n)$ into $S\mathcal{O}({2n})$.

\begin{remsdefs}\label{richiamo-Frob} \ \\
a) The pair $(SU(n) \times \mathbb{R}, p)$, where $p$ is the map $p: SU(n) \times \mathbb{R} \to U(n)$, defined by $p(B, x)= e^{{\bf i}x} B$, is the (analytic) universal covering group of $U(n)$.

Indeed $p$ is clearly an analytic homomorphism of Lie groups, whose differential, at the point $(B, x) \in SU(n) \times \mathbb{R}$, maps the tangent vector $(W, \lambda)$ to \ $e^{{\bf i}x} (W+ {\bf i} \lambda B)$. At the identity $(I_n, 0)$, this map has kernel zero and so it is an isomorphism, hence, by \cite[Prop.\,3.26 p.\,100]{Warner1983}, it is a covering map.

b) From (a), we get easily that, if $\mathcal{K}$ and $\widehat{\mathcal{K}}$ are the Killing tensors of $U(n)$ and of $SU(n) \times \mathbb{R}$, respectively, then we have: $p^*(\mathcal{K}) = \widehat{\mathcal{K}}$. Since $\widehat{\mathcal{K}}$ is the product of the Killing tensors of $SU(n)$ and of $\mathbb{R}$ (and this last is zero), then, remembering again \cite[Ex.\,6.19 p.\,129]{Sepa2007}, we have: 

$\widehat{\mathcal{K}}_{(B, x)}((W, \lambda), (W', \lambda')) = 2n \, tr(B^* W B^*W')$, 

for every $B \in SU(n)$, for every $W, W' \in T_B(SU(n))$ and for every $x, \lambda,\lambda' \in \mathbb{R}$.

Let $A:= e^{{\bf i}x} B = p(B,x)$ (with $B \in SU(n)$ and $x \in \mathbb{R}$). If $Y, Z \in T_{A}(U(n))$, then, by (a), $Y$ and $Z$ are the images, through the tangent map of $p$, of 

$\big( e^{-{\bf i} x}Y - \dfrac{tr(A^*Y)}{n}B, - \dfrac{{\bf i}}{n} tr(A^*Y) \big)$ and of $\big( e^{-{\bf i} x}Z - \dfrac{tr(A^*Z)}{n}B, - \dfrac{{\bf i}}{n} tr(A^*Z) \big)$, respectively (note that $tr(A^*Y)$ and $tr(A^*Y)$ are purely imaginary, because $A^*Y$ and $A^*Z$ are skew-hermitian matrices).
Since $p^*(\mathcal{K}) = \widehat{\mathcal{K}}$, we get that 

\smallskip

$\mathcal{K}_A(Y, Z) =  \widehat{\mathcal{K}}_{(B, x)} \big(  \big( e^{-{\bf i} x}Y - \dfrac{tr(A^*Y)}{n}B, - \dfrac{{\bf i}}{n} tr(A^*Y) \big),  \big( e^{-{\bf i} x}Z - \dfrac{tr(A^*Z)}{n}B, - \dfrac{{\bf i}}{n} tr(A^*Z) \big)  \big)$

$= 2n \, tr \big( B^*( e^{-{\bf i} x}Y - \dfrac{tr(A^*Y)}{n}B)\ B^*(e^{-{\bf i} x}Z - \dfrac{tr(A^*Z)}{n}B) \big) =$

$2n\, \big( tr (A^*Y A^*Z) -\dfrac{1}{n}tr(A^*Y)tr(A^*Z)\big)= 
2n \, tr (A^*Y A^*Z) - 2\, tr(A^*Y)tr(A^*Z) =$

$-2n \, \phi_A (Y, Z) - 2\, tr(A^*Y)tr(A^*Z)$. 

Therefore we can state the following
\end{remsdefs}

\begin{lemma}\label{espressione-K-U-n}
The Killing tensor $\mathcal{K}$ of $U(n)$ has the following expression:

$\mathcal{K}_A(Y, Z) = 2n \, tr (A^*Y A^*Z) - 2\, tr(A^*Y)tr(A^*Z) = -2n \, \phi_A (Y, Z) - 2\, tr(A^*Y)tr(A^*Z)$

for every $A \in U(n)$ and every $Y, Z \in T_A(U(n))$. 
\end{lemma}

\begin{rem}
The Killing tensor $\mathcal{K}$ of $U(n)$ is a (degenerate) negative semi-definite tensor (and so $U(n)$ is not semi-simple). 

It suffices to check it at the identity $I_n \in U(n)$. 
For, by Lemma \ref{espressione-K-U-n}, we have 

$\mathcal{K}_{I_n}({\bf i} I_n, {\bf i} I_n) = 0$ and furthermore, if $Y$ is a skew-hermitian matrix with purely imaginary eigenvalues ${\bf i} y_1, \dots , {\bf i} y_n$, then

$\mathcal{K}_{I_n}(Y, Y)=  -2n \sum_{j=1}^n y_j^2 + \sum_{h, j = 1}^n  2y_h y_j \le -2n \sum_{j=1}^n y_j^2 + \sum_{h, j=1}^n (y_h^2 + y_j^2) = 0$.
\end{rem}

\begin{remdef}\label{matrica-H}
On the product manifold $SU(n) \times \mathbb{R}$, we consider the metric $\mathcal{H}$ defined in the following way: 

$\mathcal{H}_{(B, x)} \big( (W, \lambda), (W', \lambda') \big) = - tr (B^* W B^* W')+ n \lambda \, \lambda'$, \ for every $B \in SU(n)$, for every $W, W' \in T_B (SU(n))$ and for every $x, \lambda, \lambda' \in \mathbb{R}$. 

Note that the metric $\mathcal{H}$ is the product of a constant positive multiple of the Killing metric of $SU(n)$ and of a constant positive multiple of the euclidean metric of $\mathbb{R}$.

By \cite[Thm.\,III p.\,341]{deRham1952}, the isometries of $(SU(n) \times \mathbb{R}, \mathcal{H} )$ are precisely the maps of the form \ $\Phi \times \alpha$, where $\Phi$ is an isometry  of $SU(n)$, endowed with its Killing metric, and $\alpha$ is an euclidean isometry of $\mathbb{R}$.
\end{remdef}

\begin{lemma}\label{isom-loc-U-n} \ \\
a) The map $p: (SU(n) \times \mathbb{R}, \mathcal{H}) \to (U(n), \phi)$ is a local isometry.

b) For every isometry $F$ of $(U(n), \phi)$, fixing the identity $I_n$ of $U(n)$, there is a unique isometry $\widehat{F}$ of $(SU(n) \times \mathbb{R}, \mathcal{H})$, fixing the the identity $(I_n, 0)$ of $SU(n) \times \mathbb{R}$, such that $p \circ \widehat{F} = F \circ p$. 
\end{lemma}

\begin{proof}
If $x, \lambda, \lambda' \in \mathbb{R}$, $B \in SU(n)$, $W, W' \in T_B(SU(n))$ (so $tr(B^*W)= tr(B^*W')=0$),
by Remarks-Definitions \ref{richiamo-Frob} (a), we have: 

$p^*(\phi)_{(B, x)} \big( (W, \lambda), (W', \lambda') \big) = \phi_{(e^{{\bf i}x}B)}\big( e^{{\bf i}x}(W+ {\bf i} \lambda B), e^{{\bf i}x}(W'+ {\bf i} \lambda' B) \big)=$

$-tr \big( (B^*W + {\bf i} \lambda I_n)(B^*W' + {\bf i} \lambda' I_n) \big)=-tr(B^*WB^*W') + n \lambda \lambda' = \mathcal{H}_{(B, x)} \big( (W, \lambda), (W', \lambda') \big)$, i.e. $p^*(\phi) = \mathcal{H}$ and the proof of (a) is complete.

Part (b) follows from part (a) and from the fact that $(SU(n) \times \mathbb{R}, p)$ is the universal covering of $U(n)$.
\end{proof}

\begin{prop}\label{SU-2-identity} \ \\
a) Every isometry of $(SU(2) \times \mathbb{R}, \mathcal{H})$, fixing  the identity $(I_2, 0)$ of $SU(2) \times \mathbb{R}$, projects (through the covering map $p$) as an isometry of $(U(2), \phi)$, fixing the identity $I_2$ of $U(2)$.

b) The isometries of $(U(2), \phi)$, fixing the identity $I_2$ of $U(2)$, are precisely the following maps:

$X \to BXB^*$, $X \to BX^*B^*$, $X \mapsto \dfrac{BXB^*}{\det(X)}$ and $X \to\det(X) BX^*B^*$, with $B \in SU(2)$.
\end{prop}

\begin{proof}
We denote by $id$ the identity map of $\mathbb{R}$, by $Id$ the identity map of $SU(2)$, by $j$ the inversion map of $SU(2)$ and by $C_B = L_B \circ R_{B^*}$ the inner automorphism of $SU(2)$, associated to $B$. By Remark-Definition \ref{matrica-H} and by Theorem \ref{teor-gen-n} (c), the isometries of $(SU(2) \times \mathbb{R}, \mathcal{H})$, fixing the the identity $(I_2, 0) \in SU(2) \times \mathbb{R}$, are precisely the maps of the form:

$C_B \times (\pm id) = (C_B \times id) \circ (Id \times (\pm id))$ and 

$(C_B \circ j) \times (\pm id) = (C_B \times id) \circ (j \times (\pm id))$, 

with $B \in SU(2)$.

Easy computations show that all the maps $C_B \times id$ (with $B \in SU(2)$), $Id \times id$, $Id \times (-id)$, $j \times id$ and $j \times (-id)$ project as maps of $U(2)$.
More precisely, $C_B \times id$ \ projects as the inner automorphism of $U(2)$ associated to $B$, 
$Id \times id$ as the identity map of $U(2)$, 
$Id \times (-id)$ as the involution of $U(2)$ given by $A \mapsto \dfrac{A}{\det(A)}$, 
\ $j \times id$ as the involution of $U(2)$ given by $A \mapsto \det(A) A^*$, and $j \times (-id)$ as the inversion of $U(2)$.

By composition, the maps of $U(2)$, obtained in this way, are those in the statement. They are isometries of $(U(2), \phi)$, by  Lemma \ref{isom-loc-U-n} (a). Part (b) of the same Lemma implies that there is no other isometries.
\end{proof}

\begin{thm}\label{isom-U-2}
The isometries of $(U(2), \phi)$ are precisely the following maps:

$X \to AXB$, 

$X \to AX^*B$, 

\smallskip

$X \mapsto \dfrac{AXB}{\det(X)}$,

\smallskip

$X \to\det(X) AX^*B$, 

with $A,B \in U(2)$.
\end{thm}

\begin{proof}
It follows directly from Proposition \ref{SU-2-identity} (b), by left (or right) translation with a matrix of $U(2)$.
\end{proof}

\begin{thm}\label{isom-U-n}
The isometries of $(U (n), \phi)$, with $n \ge 3$,  are precisely the following maps: 

$X \to AXB$, 

$X \to AX^{*}B$, 

$X \to A\overline{X}B$,

$X \to AX^{T}B$, 

with $A, B \in U(n)$.
\end{thm}

\begin{proof}
We use the same notations as in the proof of Proposition \ref{SU-2-identity} (with $n \ge 3$ generic, instead of $n=2$) and the further following: $\mu(X) = \overline{X}$ and $\eta(X) = X^T$ for every $X \in SU(n)$. Again, by Remark-Definition \ref{matrica-H} and by Theorem \ref{teor-gen-n} (d), the isometries of $(SU(n) \times \mathbb{R}, \mathcal{H})$, fixing the the identity $(I_n, 0)$ of $SU(n) \times \mathbb{R}$, are precisely the maps of the form:

$C_B \times (\pm id) = (C_B \times id) \circ (Id \times (\pm id))$, 

$(C_B \circ j) \times (\pm id) = (C_B \times id) \circ (j \times (\pm id))$,

$(C_B \circ \mu) \times (\pm id) = (C_B \times id) \circ (\mu \times (\pm id))$ and

$(C_B \circ \eta) \times (\pm id) = (C_B \times id) \circ (\eta \times (\pm id))$,

with $B \in SU(n)$.

Since the isometries of $(SU(n) \times \mathbb{R}, \mathcal{H})$, projecting (throughout $p$) as maps of $U(n)$, form a group with respect to the composition, it suffices to examine the following isometries: 

$C_B \times id$, which projects as the inner automorphism of $U(n)$, associated to $B \in SU(n)$,

$Id \times id$, which projects as the identity map of $U(n)$, 

$j \times (-id)$, which projects as the inversion map of $U(n)$,

$\mu \times (-id)$, which projects as the (complex) conjugation map of $U(n)$,

$\eta \times id$, which projects as the transposition map of $U(n)$,

and $Id \times (-id)$, \ $j \times id$, \ $\mu \times id$, \ $\eta \times (-id)$, which, on the contrary, do not project as maps of $U(n)$.
The proofs of the first five cases are obvious. For the isometries, which do not project as maps of $U(n)$, we consider, as example, only the case $Id \times (-id)$; the other cases can be treated in the same way.

For, we have $I_n = p(I_n, 0) = p(e^{\frac{2 \pi {\bf i}}{n}}I_n,\ - \dfrac{2 \pi}{n})$, \ \ $p \circ \big(Id \times (-id) \big)\ (I_n, 0) = I_n$ and 

\smallskip

$p \circ \big( Id \times (-id) \big)\ (e^{\frac{2 \pi {\bf i}}{n}}I_n,\ - \dfrac{2 \pi}{n}) = e^{\frac{4 \pi {\bf i}}{n}} I_n$; these last two are different, because $n \ge 3$ and so, the isometry $Id \times (-id)$ does not project as map of $U(n)$.

Therefore, taking into account Lemma \ref{isom-loc-U-n}, the isometries of $(U(n), \phi)$, fixing the identity $I_n$, are the following maps:  $X \to BXB^*$, $X \to BX^{*}B^*$,  $X \to B\overline{X}B^*$, $X \to BX^{T}B^*$, with $B \in SU(n)$. Now, by left (or right) translation with a matrix of $U(n)$, we obtain all the isometries in the statement.
\end{proof}

\begin{rems} \ \\
a) The full group of isometries of $(U(n), \phi)$, both for $n=2$ and for $n \ge 3$, has $4$ connected components, all diffeomorphic to 
$\dfrac{U(n) \times U(n)}{\{ \lambda(I_n, I_n) : \lambda \in \mathbb{C}, |\lambda|=1 \}}$.

\smallskip

Indeed, arguing as in Remark \ref{defU}, the group generated by left and right translations of $U(n)$ is diffeomorphic to $\dfrac{U(n) \times U(n)}{(Z \times Z) \cap \Delta}$, where $Z$ and $\Delta$ are, respectively, the center of $U(n)$ and the diagonal of $U(n) \times U(n)$. We conclude, because the center of $U(n)$ is $\{ \lambda I_n : \lambda \in \mathbb{C}, |\lambda|=1\}$.

b) For every $n \ge 2$, \, $(U(n), \phi)$ is a (globally) symmetric Riemannian manifold. 

Indeed, for every $A \in U(n)$, the map $X \mapsto AX^*A$ is an isometry of $(U(n), \phi)$, fixing $A$, and whose differential at $A$ is the opposite of the identity map of $T_A(U(n))$.
\end{rems}

\begin{rem}
Compare Theorems \ref{isom-U-2} and \ref{isom-U-n} with an analogous result, obtained in \cite[Thm.\,8]{HatMol2012}.
\end{rem}

\end{document}